\documentclass[12pt]{amsart}
\usepackage{enumerate}
\usepackage[hyperindex,breaklinks,colorlinks,citecolor=blue,pagebackref]{hyperref}
\usepackage[latin1]{inputenc}
\usepackage{amsmath, amsthm, amsfonts, amssymb}
\usepackage{mathrsfs}
\usepackage{graphicx}
\usepackage{latexsym}
\usepackage{fullpage}
\usepackage[all]{xy}
\usepackage{color}
\usepackage{stmaryrd}
\usepackage[mathscr]{euscript}

\newtheorem{theorem}{Theorem}[section]
\newtheorem{lemma}[theorem]{Lemma}
\newtheorem{proposition}[theorem]{Proposition}
\newtheorem{corollary}[theorem]{Corollary}

\newtheorem{definition}[theorem]{Definition}

\theoremstyle{remark}

\newtheorem{rmks}{Remarks}[section]
\newtheorem*{Examples}{Examples}

\newcommand{\B}{\mathcal{B}}
\newcommand{\C}{\mathcal{C}}

\newcommand{\F}{\mathcal{F}}

\renewcommand{\S}{\mathcal{S}}

\newcommand{\T}{\mathcal{T}}
\newcommand{\U}{\mathcal{U}}

\newcommand{\X}{\mathcal{X}}

\newcommand{\MLR}{\mathsf{MLR}}

\newcommand{\llb}{\llbracket}
\newcommand{\rrb}{\rrbracket}
\newcommand{\fr}{{^\frown}}

\newcommand{\ml}{Martin-L\"{o}f\ }

\definecolor{purple}{rgb}{.9,0.2,.9}

\newcommand{\cs}{2^\omega}
\newcommand{\ccs}{3^\omega}
\newcommand{\uh}{{\upharpoonright}}
\renewcommand{\phi}{\varphi}
\newcommand{\str}{2^{<\omega}}

\newcommand{\diverge}{{\uparrow}}

\newcommand{\ran}{\mathrm{ran}}

\title{Algorithmically Random Functions and Effective Capacities}
\author{Douglas Cenzer and Christopher P. Porter}
\begin{document}

\maketitle

\begin{abstract} We continue the investigation of algorithmically random functions and closed sets, and in particular the connection with the notion of capacity. 
We study notions of random continuous functions given in terms of a family of computable measures called symmetric Bernoulli measures.  We isolate one particular class of random functions that we refer to as random online functions $F$, where the value of $y(n)$ for $y = F(x)$ may be computed from the values of $x(0),\dots,x(n)$. We show that random online functions are neither onto nor one-to-one.  We give a necessary condition on the members of the ranges of random online functions in terms of initial segment complexity and the associated computable capacity.  Lastly, we introduce the notion of Martin-L\"of random online \emph{partial} function on $\cs$ and give a family of online partial random functions the ranges of which are precisely the random closed sets introduced in \cite{BCDW07}.

\textbf{Keywords:}  algorithmic randomness, computability theory, random closed sets, random continuous functions, capacity.
\end{abstract}

\section{Introduction}

In a series of recent papers \cite{BCDW07,BCRW09,BCRW08,BCTW11},
Barmpalias, Brodhead, Cenzer et al have developed the notion of
algorithmic randomness for closed sets and continuous functions on
$\cs$ as part of the broad program of algorithmic randomness.  The 
study of random closed sets was furthered by Axon \cite{A10}, Diamondstone and 
Kjos-Hanssen \cite{DK12}, and others.  Cenzer et al \cite{BCTW11} studied the relationship between notions of random closed sets with respect to different computable probability measures and effective capacities.  


Here we look more closely at the relationship between random continuous functions and effective capacity.  First, we generalize the notion of random continuous function from \cite{BCRW08} to a wider class of computable measures that we call symmetric Bernoulli measures.  Then we study properties of the  effective capacities associated to the classes of functions that are random with respect to various symmetric Bernoulli measures.  We isolate one such class of functions, which we refer to as random online continuous functions.  We study the reals in the range of a random online continuous function, as well as the average values of random online continuous functions.  

It turns out that a number of effective capacities cannot be generated by a class of functions that are random with respect to a symmetric Bernoulli measure.  We identify a class of measures on the space of functions that yield random online partial continuous functions and prove that a wide class of effective capacities can be generated by such functions, including the effective capacity that is associated to the original definition of algorithmically random closed set from \cite{BCDW07}. 

Algorthmic randomness for closed sets was defined in \cite{BCDW07} starting from a natural computable measure on the space $\C(\cs)$ of closed subsets of $\cs$ and using the notion of \ml randomness given by \ml tests.   It was shown that $\Delta^0_2$ random closed sets exist but there are no random $\Pi^0_1$ closed sets. It is shown that any random closed set is
perfect, has measure 0, and has box dimension $\log_2 \frac43$. A random closed set has no $n$-c.e.\ elements.

 Algorithmic randomness for continuous functions on $\cs$ was defined in 
\cite{BCRW08} by defining a representation of such functions in $3^{\omega}$ and using the uniform measure on $\cs$  to induce a measure on the space $\F(\cs)$ of continuous functions. It was shown that random $\Delta^0_2$ continuous functions exist, but no computable function can be random and no random function can map a
 computable real to a computable real. The image  of a random continuous function is always a perfect set and hence
  uncountable. For any $y \in \cs$, there exists a random continuous
  function $F$ with $y$ in the image of $F$. Thus the image of a
  random continuous function need not be a random closed set. The set
  of zeroes of a random continuous function is a random closed
  set (if nonempty). 

The connection between measure and capacity for the
space $\C(\cs)$ was investigated in \cite{BCTW11}.  For any computable measure
$\mu^*$ on $\C(\cs)$, a computable capacity may be defined by letting $T(Q)$ be the
$\mu^*$-measure of the family of closed sets $K$ which have nonempty
intersection with $Q$ for each $Q\in\C(\cs)$. An effective version of the Choquet's
theorem was obtained  by showing that every computable capacity may be obtained from
a computable measure in this way.  Conditions were given on a measure $\nu^*$ on $\C(\cs)$ that characterize when the capacity of
all $\nu^*$-random closed sets equals zero.  For certain computable measures, effectively
closed sets with positive capacity and with Lebesgue measure zero are constructed.  For computable measures, a real $q$ is upper semi-computable
if and only if there is an effectively closed set with capacity $q$.

The problem of characterizing the possible members of random closed sets was studied by Diamondstone and Kjos-Hanssen in \cite{DK12}.
They gave an alternative presentation for random closed sets and showed a strong connection between the effective Hausdorff dimension 
of a real $x$ and the membership of $x$ in a random closed set. 

The outline of the paper is as follows.  In Section \ref{sec-background}, we provide the requisite background.  In Section \ref{sec-bernoulli} we define symmetric Bernoulli measures on the
space of continuous functions on $\cs$ and prove basic facts about the domains and ranges of functions that are random with respect to such measures.  We study the connection between 
random functions and effective capacities on the space of closed subsets of $\cs$ in Section \ref{sec-functions-to-capacities}.  Next, we introduce and study the notion of a random online function in 
Section \ref{sec-online}.  Lastly, in Section \ref{sec-partial-online}, we define random online partial functions and establish a correspondence between the ranges of such functions and various families
of random closed sets.

The authors would like to thank Laurent Bienvenu and the anonymous referees for helpful comments on an earlier draft of this paper.

\section{Background}\label{sec-background}

Some definitions are needed.  For a finite string $\sigma
\in \{0,1\}^n$, let $|\sigma| = n$ denote the length of $n$. 
For two strings $\sigma,\tau$, say that $\tau$ \emph{extends} $\sigma$
and write $\sigma \prec \tau$ if $|\sigma| \leq |\tau|$ and $\sigma(i)
= \tau(i)$ for $i < |\sigma|$. For $x \in \cs$, $\sigma \prec x$ means
that $\sigma(i) = x(i)$ for $i < |\sigma|$. Let $\sigma^{\frown} \tau$
denote the concatenation of $\sigma$ and $\tau$ and let
$\sigma^{\frown} i$ denote $\sigma^{\frown}(i)$ for $i=0,1$. Let $x
\uh n = (x(0),\dots,x(n-1))$.  The empty string will be denoted $\epsilon$.
Two reals $x$ and $y$ may be coded together into $z = x \oplus y$, where 
$z(2n) = x(n)$ and $z(2n+1) =
y(n)$ for all $n$.
For a finite string $\sigma$, let $\llb\sigma\rrb$ denote $\{x \in \cs:
 \sigma \prec x \}$. We shall refer to $\llb\sigma\rrb$ as the \emph{interval}
 determined by $\sigma$. Each such interval is a clopen set and the
 clopen sets are just finite unions of intervals. 
Now a nonempty closed set $P$ may be identified with a tree $T_P
 \subseteq \{0,1\}^*$ where $T_P = \{\sigma: P \cap \llb\sigma\rrb \neq
 \emptyset\}$. Note that $T_P$ has no dead ends. That is, if $\sigma
 \in T_P$, then either $\sigma^{\frown}0 \in T_P$ or $\sigma^{\frown}1
 \in T_P$ (or both).
For an arbitrary tree $T \subseteq \{0,1\}^*$, let $[T]$ denote the
set of infinite paths through $T$.
It is well-known that $P \subseteq \cs$ is a closed set if and only if
$P = [T]$ for some tree $T$.  $P$ is a $\Pi^0_1$ class, or an effectively
closed set, if $P = [T]$ for some computable tree $T$. 


%

A measure $\nu$ on $\cs$ is \emph{computable} if there is a computable
function $\hat\nu:\str\times\omega\rightarrow\mathbb{Q}_2$ (where $\mathbb{Q}_2=\{\frac{m}{2^n}:n,m\in\omega\}$)
such that $|\nu(\llb\sigma\rrb)-\hat\nu(\sigma,i)|\leq 2^{-i}$ for every $\sigma\in\str$ and $i\in\omega$.  A computable measure
on $\ccs$ is similarly defined.


Martin-L\"{o}f \cite{ML66} observed that stochastic properties could
be viewed as special kinds of effectively presented measure zero sets and defined a random
real as one that avoids these measure 0
sets. More precisely, a real $x \in \cs$ is Martin-L\"{o}f random if for every
effective sequence $S_1, S_2, \dots$ of c.e.\ open sets with $\mu(S_n)
\leq 2^{-n}$, $x \notin \bigcap_n S_n$ (where $\mu$ is the uniform
measure on $\cs$).   This can be straightforwardly extended to any computable measure 
$\nu$ on $\cs$ or $\ccs$ by replacing the condition $\mu(S_n)
\leq 2^{-n}$ with $\nu(S_n)\leq 2^{-n}$.


%

Given a measure $\mu$ on $3^{\omega}$, we define a measure $\mu^*$ on the space $\C(\cs)$
of closed subsets of $\cs$ as follows. Given a closed set $Q \subseteq
\cs$, let $T = T_Q$ be the tree without dead ends such that $Q =[T]$. Let
$\sigma_0, \sigma_1, \ldots$ enumerate the elements of $T$ in order,
first by length and then lexicographically. We then define the 
{\em (canonical) code} $x = x_Q = x_T$ of $Q$ by recursion such that for each $n$, $x(n) =2$ if both
$\sigma_n\fr 0$ and $\sigma_n\fr 1$ are in $T$, $x(n) =1$ if
$\sigma_n\fr 0 \notin T$ and $\sigma_n\fr 1 \in T$, and $x(n)
=0$ if $\sigma_n\fr 0 \in T$ and $\sigma_n\fr 1 \notin T$. We
then define $\mu^*$ by setting
\begin{equation}
\mu^*({\mathcal X}) = \mu(\{x_Q:Q \in {\mathcal X}\})
\end{equation}
for any ${\mathcal X} \subseteq {\mathcal C}(\cs)$.  For the uniform measure, this means that
given $\sigma \in T_Q$, there is probability $\frac13$ that both
$\sigma^{\frown}0 \in T_Q$ and $\sigma^{\frown}1 \in T_Q$ and, for
$i=0,1$, there is probability $\frac13$ that only $\sigma^{\frown}i
\in T_Q$.  Brodhead, Cenzer, and Dashti \cite{BCDW07} 
defined a closed set $Q \subseteq \cs$ to be 
(Martin-L\"{o}f) random if $x_Q$ is  (Martin-L\"{o}f) random.  We will sometimes refer to the random
closed sets given by the uniform measure on $\ccs$ as the \emph{standard random closed sets}.

Given a continuous function $F$ on $\cs$, observe that for any $\sigma\in\str$ there is some $n\in\omega$ and $\tau\in\str$ of length $n$ such that for all $x\in\llb\sigma\rrb$, $F(x)\uh n=\tau$.

Let $\F(\cs)$ denote the collection of all continuous functions $F:\cs\rightarrow\cs$. Each $F\in\F(\cs)$ may be represented by a function $f: \str\setminus\{\epsilon\} \to \{0,1,2\}$, defined inductively as follows.  
Suppose we have defined $f(\sigma\uh i)=e_i$ for $i=1,\dotsc, n$ and every $\sigma$ of length $n$.  Then given some $\sigma$ of length $n+1$, where $f(\sigma\uh i)=e_i$ for $i=1,\dotsc, n$,
 let $\rho=(n_1,\dotsc,n_k)$ be the result of deleting all 2s from $(e_1,\dotsc,e_n)$. If for all $x\in\llb\sigma\rrb$, $F(x)\uh (k+1) = \rho^\frown j$ for some  $j\in\{0,1\}$, then we may set $e_{n+1}=j$, although we may set $e_{n+1}=2$.  If there is no such $j$, we must set $e_{n+1}=2$.  It is helpful to think of the 2's as delaying the output of $F$ along initial segments of some $x\in\cs$.  For each $F\in\F(\cs)$, there are infinitely many functions that represent $F$, and $f: \str\setminus\{\epsilon\} \to \{0,1,2\}$ defines a (possibly partial) $F\in\F(\cs)$.

%


%
Each representing function $f: \str\setminus\{\epsilon\} \to \{0,1,2\}$ can be straightforwardly coded as some $z\in\ccs$.
We can thus define a measure $\mu^{**}$ on $\F(\cs)$ induced 
by the uniform measure on $\ccs$.   As with the case of computable measures
on $\C(\cs)$, every computable measure $\nu$ on $\ccs$ induces a computable measure
$\nu^{**}$ on $\F(\cs)$.
 Brodhead, Cenzer, and Remmel  \cite{BCR07} defined $F\in\F(\cs)$ to be Martin-L\"of random
if $F$ is represented by a representing function coded by a Martin-L\"of random $z\in\ccs$.  We will sometimes refer to the random
continuous functions given by the uniform measure on $\ccs$ as the \emph{standard random continuous functions}.

Next we consider the notion of a capacity. 

\begin{definition}
A \emph{capacity} on $\C(\cs)$ is a function $\T: \C(\cs) \to [0,1]$ with
$\T(\emptyset) =0$ such that
\begin{enumerate}
\item $\T$ is monotone increasing, that is, $Q_1 \subseteq Q_2$ implies $\T (Q_1) \leq \T(Q_2)$.
\item $\T$ has the \emph{alternating of infinite order} property, that is,
for $n \geq 2$ and any $Q_1, \dots, Q_n \in \C$
\[
\T(\bigcap_{i=1}^n Q_i) \leq \sum \{(-1)^{|I|+1} \T(\bigcup_{i \in
  I}Q_i): \emptyset \neq I \subseteq \{1,2,\dots,n\} \}.
\]
\item If $Q = \bigcap_n Q_n$ and $Q_{n+1} \subseteq Q_n$ for all
$n$, then $\T(Q) = \lim_{n \to \infty} \T(Q_n)$.
\end{enumerate}
\end{definition}

We will also assume, unless otherwise specified, that 
$\T(\cs) = 1$.  We will say that a capacity $\T$ is computable if it is computable on
the family of clopen sets, that is, if there is a computable function $F$ from 
the Boolean algebra $\B$ of clopen sets into $[0,1]$ such that 
$F(B) = \T(B)$ for any $B \in \B$. 

Given a measure $\mu^*$ on the space $\C(\cs)$ of closed sets, define 
\[
\T_{\mu}(Q)=\mu^{*}(\{\X\in\C(\cs):\X\cap\S\neq\emptyset\}),
\]
That is, $\T_{\mu}(Q)$ is the probability that a randomly chosen closed set meets $Q$. The following effective version
of the Choquet Capacity Theorem was shown in \cite{BCTW11}.

\begin{theorem}[\cite{BCTW11}] \label{thchc} \begin{enumerate}
\item For any computable  probability measure $\mu$  on $\C(\cs)$, $\T_{\mu}$ is a computable capacity.

\item For any computable capacity $\T$ on $\C(\cs)$, there is a computable measure
$\mu$ on the space of closed sets such that $\T =
\T_{\mu}$. 
\end{enumerate}
\end{theorem}

For a given computable capacity $\T$, if $\mu^*$ is a computable measure on $\C(\cs)$ such that $\T=\T_\mu$, we will refer to $\mu^*$-random closed sets as the random closed sets associated to $\T$ and $\T$ as the capacity associated to the $\mu^*$-random closed sets.

\section{Symmetric Bernoulli Measures on $\F(\cs)$}\label{sec-bernoulli}

In this section, we consider continuous functions that are random with respect to some measure from a specific class of computable measures on $\ccs$.

\begin{definition} Let $\mu$ be a measure on $\ccs$.  
\begin{itemize}
\item[(i)] $\mu$ is a \emph{Bernoulli} measure if there are $p_0,p_1, p_2\in[0,1]$ such that $p_0+p_1+p_2=1$ and $\mu(\sigma^\frown i)=p_i\cdot\mu(\sigma)$ for each $i\in\{0,1,2\}$.
\item[(ii)] $\mu$ is a \emph{symmetric} Bernoulli measure if $\mu$ is a Bernoulli measure and there is some $r\in[0,1/2]$ such that $r=p_0=p_1$ (so that $p_2=1-2r$). \end{itemize}
\end{definition}

The symmetric Bernoulli measure with parameter $r\in[0,1/2]$ will be denoted $\mu_r$.  Note that $\mu_r$ is computable if and only if $r$ is a computable real number.  

We are interested in the behavior of the $\mu_r^{**}$-random continuous functions on $\cs$.  Note that in the case that $r=1/3$, $\mu_r$ is the uniform measure on $\ccs$ and the $\mu_r^{**}$-random continuous functions are the standard random continuous functions discussed in the previous section.  In fact, the results in this section generalize certain results from \cite{BCRW09} concerning $\mu_{1/3}^{**}$-random continuous functions.

First, it was shown in \cite{BCRW09} that every $\mu_{1/3}^{**}$-random continuous function is total.  However, if we allow the parameter $r$ to vary, which results in a change of the probability of the occurrence of delays (i.e., the occurrence of 2s), the situation becomes slightly more interesting.  Specifically, if $\mu_r$ is such that the probability of delay is greater than or equal to $1/2$, then not every $\mu_r^{**}$-random function will be total.

The following lemma will be needed. 

\begin{lemma} \label{lem:key}
Let $\mu_r$ be a symmetric Bernoulli measure on $\ccs$, let $A \subseteq \{0,1,2\}$, and let $p = \sum_{i \in A} p_i$, where $p_0=p_1=r$ and $p_2=1-2r$. Then the $\mu^{**}$-measure $q$ of the functions $F\in\F(\cs)$ such that there exists $x \in \cs$ with $f(x \uh n) \in A$ for all $n$ (where $f$ is the function representing $F$) equals 0 if $p \leq1/2$ and equals $\frac{2p-1}{p^2}$ if $p > 1/2$. 
\end{lemma}

\begin{proof} It follows from the compactness of $\cs$ that there exists $x$ such that $f(x \uh n) \in A$ for all $n>0$ if and only if for every $n$, there exists $\sigma \in \{0,1\}^n$ such that $f(\sigma \uh m) \in A$ for all $0<m < n$.  Let $q_n$ be the probability that such $\sigma \in \{0,1\}^n$ exists. Then $q_0 = 1$, $q_{n+1} \leq q_n$ for all $n$, and $q = \lim_{n \to \infty} q_n$. 
Considering the cases of $f(i)$ for $i \in \{0,1\}$, we calculate that
\[
q_{n+1} = 2p q_n - p^2 q_n^2.
\]
Taking the limit of both sides, we see that $q = 2pq - p^2q^2$, so that either $q = 0$ or $q = \frac{2p-1}{p^2}$. In the case that $p < 1/2$, the latter is negative. Thus $q = 0$ if 
$p \leq 1/2$.  

For the other case, note first that $2pq_n - p^2 q_n^2 = 1 - (1 - pq_n)^2$, so that $q _n \geq x$ implies that $2pq_n - p^2 q_n^2  \geq 2px - p^2 x^2$. Let $s = \frac{2p-1}{p^2}$.
We now show by induction that $q_n \geq s$ for all $n$.  Initially we have $q_0 = 1 \geq s$. Now assuming that $q_n \geq s$, 
it follows that 
\[
q_{n+1} = 2pq_n - p^2 q_n^2  \geq 2p s - p^2 s^2 = s (2p - p^2s) = s (2p - (2p-1)) = s.
\]
Now suppose that $p > 1/2$, so that $s = \frac{2p-1}{p^2} > 0$. 
Since the sequence $(q_n)_{n \in \omega}$ is decreasing and $q_n \geq s$ for all $n$, it follows that the limit $q = \lim_n q_n \geq s$ and hence $q = s$.
\end{proof}

\begin{proposition}\label{prop:prob-partial}
Let $\mu_r$ be a symmetric Bernoulli measure on $\ccs$ for some $r\in[0,1/2]$.  Then the $\mu_r^{**}$-measure of the collection of partial continuous functions on $\cs$ is 0 if $r\geq 1/4$ and is 1 if $r<1/4$.
\end{proposition}

\begin{proof}
First note that the measure must be either 0 or 1 in either case. This is because a function $F$ is total if and only if the restrictions of $F$ to both $\llb 0\rrb$ and $\llb1\rrb$ are total, so that if $p$ is the measure of the set of total functions, then $p = p^2$.  Next observe that the function represented by $f:\str\rightarrow \{0,1,2\}$  is partial if and only if there exists $x \in \cs$ and $n$ such that $f(x \uh m) = 2$ for all $m \geq n$. It  is enough to compute the probability $q$ that there exists $x$ such that $f(x \uh m) = 2$ for all $m>0$.  

Let $A=\{2\}$, so that $f(\sigma)\in A$ with probability $p=1-2r$ for each $\sigma\in\str\setminus\{\epsilon\}$.  Then by Lemma \ref{lem:key}, the $\mu^{**}$-measure of functions $F$ such that there exists $x\in\cs$ with $f(x \uh n) \in A$ for all $n>0$ equals 0 if $r\geq 1/4$ and equals $\frac{2p-1}{p^2}=\frac{1-4r}{(1-2r)^2}$ if $r < 1/4$.  Since for $r<1/4$, there are positive $\mu^{**}$-measure many functions $F$ for which such an $x$ exists, it follows that the collection of partial functions has $\mu^{**}$-measure 1.


\end{proof}


Next, it was also shown in \cite{BCRW09} that the probability that the range of a random continuous function includes a fixed $y\in\cs$ is equal to 3/4.  This was obtained by computing, for each $\sigma\in\str$ of length $n$, the probability $p_n$ that the range of a random continuous function has non-empty intersection with $\llb\sigma\rrb$ and then proving that $\lim_{n\rightarrow\infty}p_n=3/4$.  We consider the analogous result in the general case of a symmetric Bernoulli measure.

\begin{theorem}\label{thm:hitting-prob}
Let $\mu_r$ be a symmetric Bernoulli measure on $\ccs$ for some $r\in(0,1/2]$ and let $y\in\cs$.  Then the $\mu_r^{**}$-measure of the collection of continuous functions $F$ such that $y\in\ran(F)$ is equal to 
\[
\dfrac{1-2r}{(1-r)^2}.
\]
\end{theorem}

\begin{proof}
By symmetry of the measure $\mu_r$, it suffices to show that $\mu_r^{**}$-measure of the collection of continuous functions $F$ such that $0^\infty\in\ran(F)$ is equal to $\frac{1-2r}{(1-r)^2}$.  Let $A=\{0,2\}$, so that $f(\sigma)\in A$ with probability $p=1-r$ for $\sigma\in\str\setminus\{\epsilon\}$.  Then by Lemma \ref{lem:key}, the $\mu^{**}$-measure of functions $F$ such that there exists $x\in\cs$ with $f(x \uh n) \in A$ for all $n$ equals 0 if $r\geq 1/2$ and equals $\frac{2p-1}{p^2}=\frac{1-2r}{(1-r)^2}$ if $r < 1/2$.   

Note that even if a function $F$ satisfies $f(x\uh n)\in A$ for every $n>0$ for some $x\in\cs$, this does not guarantee that $0^\infty\in\ran(F)$, since we may have $f(x\uh n)=2$ for all but finitely many $n$.  For a given $F\in\F(\cs)$, let $\C_F=\{x\in\cs: (\forall n) f(x\uh n)\in A\}$.  One can verify that the probability that $0^\infty\in\ran(F)$, given that $\C_F$ is non-empty, is 1 as follows.  Suppose that $\C_F$ is non-empty.  Then if we consider the left-most path $x$ of $\C_F$, by the law of large numbers, as the occurrence of the label 0 on initial segments of $x$ is $\frac{r}{1-r}$, the limiting frequency of 0s along $x$ is $\frac{r}{1-r}$ with probability 1.  Since the $\mu_r^{**}$-measure of the collection of functions $F$ such that $\C_F$ is non-empty is $\frac{1-2r}{(1-r)^2}$, the conclusion follows.

\end{proof}

Observe that as $r$ approaches 0, the above probability approaches 1.  This means that as the probability of delay approaches 1, we have more chances to hit any given real, and so this probability approaches one.  However, for the value $r=0$, we have a discontinuity, as the resulting measure is concentrated on the function coded by $2^\infty$, which never outputs any bits but only delays indefinitely on every possible input.  Lastly, as $r$ approaches 1/2, the above probability approaches 0. In fact, this probability only attains the value 0 when $r=1/2$, that is, when the $\mu_r^{**}$-random functions have no delay.  Hereafter, we will refer to $\mu_{1/2}^{**}$-random functions as \emph{random online functions}, which we study in detail in Section \ref{sec-online}.

\section{From Functions to Capacities}\label{sec-functions-to-capacities}

The significance of the proof of Theorem \ref{thm:hitting-prob} is that it reveals a connection between a notion of random continuous function and a notion of effective capacity.  In particular, we have the following result.

\begin{theorem}\label{thm-fns-caps}
Let $\nu^{**}$ be a computable measure on $\F(\cs)$ and suppose that every $\nu^{**}$-random function is total.  Then the function
\[
\T(\S)=\nu^{**}(\{F\in\F(\cs):\ran(F)\cap\S\neq\emptyset\})
\]
is a computable capacity on $\C(\cs)$.
\end{theorem}

\begin{proof}
First we show that the map taking a $\nu^{**}$-random function to its range induces a computable measure on $\C(\cs)$.  Let $F$ be a $\nu^{**}$-random function.  Since $F$ is a continuous map from a compact space to a Hausdorff space, $F$ is a closed map.  By assumption, $F$ is total, and hence $\ran(F)=F(\cs)$ is a closed set.  Moreover, it is not hard to see that there is a (partial) Turing functional $\Phi:\ccs\rightarrow\ccs$ that, given a real in $\ccs$ that codes a representing function $f$ of some $\nu^{**}$-random function $F$, outputs a real that codes the range of $F$.  One can verify that $\Phi$ is defined on a subset of $\ccs$ of $\nu$-measure one.  It follows that $\Phi$ and $\nu$ together induce a computable measure $\nu_\Phi$ on $\ccs$ defined by
\[
\nu_\Phi(\X)=\nu(\Phi^{-1}(\X))
\]
for all measurable $\X\subseteq\ccs$ (see \cite[Lemma 2.6]{BP12}).  It follows from the preservation of randomness theorem (\cite[Theorem 3.2]{BP12}) that the image of a $\nu$-random real under $\Phi$ is a $\nu_\Phi$-random real.  In addition, by the no randomness ex nihilo principle (\cite[Theorem 3.5]{BP12}), every $\nu_\Phi$-random real is the image of a $\nu$-random real under $\Phi$.  Thus, it follows that the range of a $\nu^{**}$-random continuous function is a $\nu_\Phi^{*}$-random closed set and every $\nu_\Phi^{*}$-random is in the range of some $\nu^{**}$-random continuous function.

Thus we have
\[
\T(Q)=\nu^{**}(\{F\in\F(\cs):\ran(F)\cap Q\neq\emptyset\})=\nu_\Phi^{*}(\{C\in\C(\cs):\C\cap Q\neq\emptyset\})
\]
for every $Q\in\C(\cs)$.  By the Theorem \ref{thchc}, it follows that $\T$ is a computable capacity.
\end{proof}


%

In the proof of Theorem \ref{thm-fns-caps}, we showed that if $\nu^{**}$ is a computable measure on $\F(\cs)$ such that the $\nu^{**}$-random functions are total, then the ranges of the $\nu^{**}$-random functions yield a notion of random closed sets with respect to some computable measure $\nu^*_\Phi$ on $\C(\cs)$.   This raises the following question:  Is there a computable measure $\nu^{**}$ on $\F(\cs)$ such that the ranges of the $\nu^{**}$-random functions are the standard random closed sets?

We will provide a full answer to this question in Section \ref{sec-partial-online}, but as a first step, we prove the following.

\begin{proposition}\label{prop-functions-and-racs}
Let $\mu_r$ be a symmetric Bernoulli measure on $\ccs$ with $r\in(0,1/2)$.  Then the collection of ranges of the $\mu_r^{**}$-random functions is not the collection of standard random closed sets.
\end{proposition}

\begin{proof}
Let $r\in(0,1/2)$.  By Theorem \ref{thm:hitting-prob}, the $\mu_r^{**}$-measure of the collection of continuous functions $F$ such that $0^{\infty}\in\ran(F)$ is equal to 
$\dfrac{1-2r}{(1-r)^2}>0$.  However, as shown in \cite{BCDW07}, no standard random closed set contains a computable real, and thus the conclusion follows.
\end{proof}

A more significant difference between the collection of ranges of the $\mu_r^{**}$-random functions and the collection of standard random closed sets can be seen by considering the computable capacity associated to each of these two collections.  First, let $\mu$ be the uniform measure on $\ccs$. Then the capacity $\T_\mu(Q)$ on $\C(\cs)$ associated to the collection of standard random closed sets (see Theorem \ref{thchc}) can be shown to satisfy $\T(\llb\sigma\rrb)=\bigl(\frac{2}{3}\bigr)^n$ for every $n\in\omega$ and every $\sigma\in\str$ of length $n$.  Thus for $x\in\cs$, $T_\mu(\{x\})=\lim_{n\rightarrow\infty}T_\mu(\llb x\uh n\rrb)=0$.

Now suppose that $r\in(0,1/2)$.  Let $\nu=\mu_r$ and let $\T_r$ be the capacity from Theorem \ref{thm-fns-caps}.  Then as we proved $\T_r(\{x\})>0$ for every $x\in\cs$.   Thus, if we want to find a family of random functions such that the ranges of all such functions are the standard random closed sets, then we need the capacity $\T$ associated to this family to satisfy $\T(\{x\})=0$ for every $x\in\cs$.

One such candidate is the collection of $\mu_{1/2}^{**}$-random functions, for by Theorem \ref{thm-fns-caps}, in the case that $r=1/2$, we have $\T_r(\{x\})=0$ for every $x\in\cs$.  Is it the case that the ranges of the $\mu_{1/2}^{**}$-random functions are the standard random closed sets?  To answer this question, we will look more closely at the $\mu_{1/2}^{**}$-random functions.

%


%
%

\section{Random Online Functions}\label{sec-online}

In this section, we study the collection of functions that are random with respect to the measure $\mu^{**}_{1/2}$ induced by the symmetric Bernoulli measure $\mu_{1/2}$ on $\ccs$.  We will hereafter refer to the $\mu_{1/2}^{**}$-random functions as the \emph{random online functions} due to the absence of $2$'s in their codes in $\ccs$, which means that each bit given as input to such a function immediately (and randomly) yields one bit as output.  Given this absence of 2s, we can equivalently define a random online function to be given by a representing function $f:\str\setminus\{\epsilon\}\rightarrow\{0,1\}$.  In this case, each online function 
has precisely one representing function.  To see this, let $(\sigma_n)_{n\in\omega}$ be the canonical listing of $\str$ in length-lexicographical order.  Then given $x\in\cs$, we define a representing function $f_x$ such that $f_x(\sigma_{n+1})=x(n)$ for every $n\in\omega$.  One can readily verify that the function $F_X$ defined by
\[
F_x(y)=f_x(y\uh 1)^\frown f_x(y\uh 2)^\frown f_x(y\uh 3)^\frown\dotsc
\]
is an online function, and that every online function can be obtained in this way.  Thus, a function $F\in\F(\cs)$ is a random online function if and only if $F$ has a representing function $f$ coded by a Martin-L\"of random $x\in\cs$.

Note that by Proposition \ref{prop:prob-partial}, every random online function is total.  We establish several additional results.

\begin{theorem}
No computable real is in the range of a random online function.
\end{theorem}

\begin{proof}
The proof can be obtained by modifying the proof of Theorem 2.4 from \cite{BCRW09}, according to which no standard random continuous function is partial.\end{proof}

\begin{corollary}
No random online function is onto.
\end{corollary}

\begin{theorem}\label{thm-online-not-onto}
Let $F$ be a random online function and let $x\in\cs$ code the representing function of $F$.  If $y$ is Martin-L\"of random with relative to $x$, then $F^{-1}(\{F(y)\})$ is a standard random closed set.
\end{theorem}

\begin{proof}[Sketch]
We define a map $\Theta:\cs\rightarrow\ccs$ such that maps the join of two reals $x\oplus y\in\cs$ to some $z\in3^\omega$, where $x$ is the code of the representing function of a random online function and $z$ is a code of the closed set $F^{-1}(\{F(y)\})$.  One can verify that $\Theta$ induces the uniform measure on $\ccs$.  Given $y\in\MLR^x$, by van Lambalgen's theorem (see \cite[Theorem 6.9.1]{DH11}) the real $x\oplus y$ is random, and hence by the preservation of randomness theorem, $\Theta(x\oplus y)=z$ is random with respect to the measure induced by $\Theta$, namely the uniform measure on $\cs$, which establishes the theorem.
\end{proof}

\begin{corollary}
No random online function is one-to-one.
\end{corollary}

\begin{proof}
Given a random online function $F$, let  let $x\in\cs$ code the representing function of $F$.  Since $F$ is total, $F$ is defined on some $y$ that is Martin-L\"of random relative to $x$.  Then by Theorem \ref{thm-online-not-onto}, $F^{-1}(\{F(y)\})$ is a random closed set, which is perfect (as shown in \cite{BCDW07}).  Thus $F$ is not one-to-one.
\end{proof}

By Theorem \ref{thm:hitting-prob}, for a fixed $y\in\cs$, the probability that a random online function will have $y$ in its range is 0.  In fact, if for each $n$ we let $p_n$ be the probability that a random online function hits $\llb\sigma\rrb$ for a fixed $\sigma$ of length $n$ (where $F$ \emph{hits} $\llb\sigma\rrb$ if $\ran(F)\cap\llb\sigma\rrb\neq\emptyset$), by considering the cases of $f(i)$ for $i \in \{0,1\}$, one can show that 
\begin{itemize}
\item[(i)] $p_1=3/4$, and
\item[(ii)] $p_{n+1}=p_n(1-\frac{1}{4}p_n)$.
\end{itemize}
Moreover, one can verify that $\lim_{n\rightarrow\infty}p_n=0$ for each $n\geq 1$.  Hereafter, we will refer to the $p_i$'s as \emph{hitting probabilities}.

Using the notation of the previous section, it follows that $\T_{1/2}(\sigma)=p_n$ for every $n$ and every $\sigma$ of length $n$.  We can use this fact to determine the computable measure $\nu$ on $\ccs$ with the property that the $\nu^*$-random closed sets are precisely the ranges of random online functions.  Following the proof of the effective Choquet capacity theorem from \cite{BCTW11} to find the values of $\nu$, the key observation to make is that for each $n\in\omega$ and each $\sigma\in \str$ of length $n$,
\[
\nu(\sigma2\mid\sigma)=2\Bigl(\frac{p_n}{p_{n-1}}\Bigr)-1=2(1-\frac{1}{4}p_n)-1=1-\frac{1}{2}p_n
\]
for $n\geq 1$ (where $p_0=1$).  Here $\nu( \sigma i\mid\sigma)$ is the probability, under $\nu$,  that a random function $F$ hits $\llb\sigma i\rrb$ given that $F$ hits $\llb\sigma\rrb$. For each such $\sigma$, we thus have $\nu(\sigma 0\mid\sigma)=\nu(\sigma 1\mid\sigma)=\frac{1}{4}p_n$.  Since $\lim_{n\rightarrow\infty}p_n=0$, $\nu(\sigma 2\mid \sigma)$ approaches 1 while $\nu(\sigma 0\mid\sigma)$ and $\nu( \sigma 1\mid\sigma)$ both approach 0 as we consider longer and longer strings $\sigma$.  Thus one can prove:

\begin{theorem}\label{thm-online-range}
For each random online function $F$, the range of $F$ is not a standard random closed set.
\end{theorem}

\begin{proof}
Let $\mu$ be the uniform measure on $\ccs$ and let $\nu$ be the measure on $\ccs$ as defined above.  Then one can verify that $\mu/\nu$ is a computable $\nu$-martingale on $\ccs$, where $d:\str\rightarrow [0,+\infty)$ is a $\nu$-martingale on $\ccs$ if
\[
\nu(\sigma)d(\sigma)=\nu(\sigma0)d(\sigma0)+\nu(\sigma1)d(\sigma1)+\nu(\sigma2)d(\sigma2).
\]
\noindent Given a $x\in\ccs$, for each $n\geq 0$ we can write
\[
\frac{\mu\bigl(x\uh (n+1)\bigr)}{\nu\bigl(x\uh (n+1)\bigr)}=\frac{\mu\bigl(x\uh (n+1)\mid x\uh n\bigr)}{\nu\bigl(x\uh (n+1)\mid x\uh n\bigr)}\frac{\mu(x\uh n)}{\nu(x\uh n)},
\]
 Since $\lim_{n\rightarrow\infty}p_n=0$, for each $k$, there is some $n_k$ such that $p_{n_k}\leq 2^{-k}$.  Then for any $\sigma$ of length greater than $n_k$, we have $1\geq\nu(\sigma2\mid \sigma)\geq1-2^{-(k+1)}$ and $\nu(\sigma0\mid\sigma)=\nu(\sigma1\mid\sigma)\leq 2^{-(k+2)}$.   If $x\in\ccs$ is $\mu$-random, then for each $n\geq n_k$ such that $x(n)=2$, which happens roughly 1/3 of the time, we have
\[
\frac{\mu\bigl(x\uh (n+1)\bigr)}{\nu\bigl(x\uh (n+1)\bigr)}=\frac{1/3}{\nu\bigl({(x\uh n)}^\frown 2\mid x\uh n\bigr)}\frac{\mu(x\uh n)}{\nu(x\uh n)}\geq1/3\frac{\mu(x\uh n)}{\nu(x\uh n)},
\]
For each $n\geq n_k$ such that $x(n)=0$ or $x(n)=1$, which happens roughly 2/3 of the time, we have for $i=0,1$,
\[
\frac{\mu\bigl(x\uh (n+1)\bigr)}{\nu\bigl(x\uh (n+1)\bigr)}=\frac{1/3}{\nu\bigl({(x\uh n)}^\frown i\mid x\uh n\bigr)}{\nu(x\uh n)}\geq\frac{1/3}{2^{-(k+2)}}\frac{\mu(x\uh n)}{\nu(x\uh n)}\geq 2^k\frac{\mu(x\uh n)}{\nu(x\uh n)}.
\]
One can verify that $\lim_{n\rightarrow\infty}\frac{\mu(x\uh n+1)}{\nu(x\uh n+1)}=\infty$ for every $\mu$-random $x\in\ccs$.  It is well-known that this implies that no such $x$ can be $\nu$-random, and the conclusion follows.
\end{proof}

It is reasonable to ask which reals are in the range of some random online function.  We give a partial answer to this question by providing a necessary condition for being a member of the range of some random online function.  We first prove a more general result, which is an extension of a result in \cite{DK12}, according to which every member of a standard random closed set must have sufficiently high effective Hausdorff dimension.  Recall that $K(\sigma)$ is the prefix-free Kolmogorov complexity of $\sigma$.

\begin{theorem}\label{thm:isc-capacity}
Let $\mu^*$ be a computable measure on  $\C(\cs)$ and $\T_\mu$ the computable capacity associated to $\mu$.
If $x$ is a member of some $\mu^*$-random closed set, then there is some $c$ such that
\[
K(x\uh n)\geq-\log\T_\mu(\llb x\uh n\rrb)-c
\]
for all $n$.
\end{theorem}

\begin{proof}
Suppose that $x$ is such that for every $c$, there is some $n$ such that
\[
K(x\uh n)<-\log\T_\mu(\llb x\uh n\rrb)-c.
\]
We first define
\[
S_i=\{\sigma\in\str:K(\sigma)<-\log\T_\mu(\llb\sigma\rrb)-i\}.
\]
Next, we let $\widehat S_i$ consist of those strings in $S_i$ with no proper initial segments in $S_i$, so that $\llb\widehat S_i\rrb=\llb S_i\rrb$.
Lastly, we define
\[
\U_i=\{Q\in\C(\cs):(\exists\sigma\in\widehat S_i)[Q\cap\llb\sigma\rrb\neq\emptyset]\}.
\]
Then
\[
\mu^*(\U_i)\leq\sum_{\sigma\in\widehat S_i}\mu^*(\{Q\in\C(\cs):Q\cap\llb\sigma\rrb\}\neq\emptyset)=\sum_{\sigma\in\widehat S_i}\T_\mu(\llb\sigma\rrb)<\sum_{\sigma\in \widehat S_i}2^{-K(\sigma)-i}\leq 2^{-i},
\]
where the last inequality follows from the fact that $\sum_{\sigma\in\str}2^{-K(\sigma)}\leq 1$.  Thus, $(\U_i)_{i\in\omega}$ forms a $\mu^*$-Martin-L\"of test.  Now let $Q\in\C(\cs)$ be such that $x\in Q$.  Then for each $i$, there is some least $n$ such that $x\uh n\in \widehat{S}_i$, and thus $Q\in\U_i$.  It follows that no $Q\in\C(\cs)$ containing $x$ is $\mu^*$-random.
\end{proof}

An order function $f:\omega\rightarrow\omega$ is a non-decreasing, unbounded function.  Recall further that a real $x\in\cs$ is \emph{complex} if there is some computable order function $f$ such that $K(x\uh n)\geq f(n)$ for every $n$.  Let $(p_n)_{n\in\omega}$ be the collection of hitting probabilities determined by the collection of random online functions.   Since $(p_n)_{n\in\omega}$ is a computable, strictly decreasing sequence of rationals that converges to 0, it follows that the function $f(n)=-\log p_n$ is a computable order function.

This observation, combined with Theorem \ref{thm:isc-capacity}, yields:
\begin{corollary}
If $x\in\cs$ is in the range of a random online function, then 
\[
K(x\uh n)\geq-\log p_n-c
\]
for some $c\in\omega$.  In particular, $x$ is complex.
\end{corollary}
%

\noindent We conjecture that the converse, or some minor variant thereof, holds as well.

\section{Random Online Partial Functions}\label{sec-partial-online}

As we have seen, for each symmetric Bernoulli measure $\mu_r$ on $\cs$ with $r\in(0,1/2)$, the collection of ranges of the $\mu_r^{**}$-random functions is not the collection of standard random closed sets.  The collection of ranges of random online functions was, at first glance, a reasonable candidate for being equal to the collection of standard random closed sets, but this too fails by Theorem \ref{thm-online-range}.  Thus, we cannot use symmetric Bernoulli measures to obtain such a class of random functions.

As discussed in Section \ref{sec-functions-to-capacities},  the capacity $\T$ associated to the standard random closed sets satisfies $\T(\{x\})=0$ for every $x\in\cs$.  Thus, for any collection of random functions the ranges of which are the standard random closed sets, we need the capacity associated with this collection of functions to converge to zero quickly.  Note, however, that by Theorem \ref{thm:hitting-prob}, as we increase the possibility of delay in our functions, this actually increases the probability that we hit a given real.  

The first step to a solution is to introduce a notion of random online \emph{partial}  function.  As with the representing functions of continuous functions on $\cs$, we define an online partial function to be given by a $\{0,1,2\}$-valued representing function. The values 0 and 1 play the same role as before, but the 2's play a different role.  If $F$ is the partial function given by a $\{0,1,2\}$-valued representing function $f$, for each $\sigma\in\str$ with $f(\sigma)=2$, we have $F(X)\diverge$ for every $X\succ\sigma$.  That is, instead of causing our function to delay at a given node, a node labelled with a `2' indicates that our function is undefined on all reals extending this node.

Observe that each symmetric Bernoulli measure $\mu_r$ on $\ccs$ yields a notion of random online partial function.  However, for certain choices of $r$, we are not even guaranteed to have any functions with non-empty domain.

\begin{proposition}
If $\mu_r$ is a computable symmetric Bernoulli measure on $\ccs$, then the probability that a $\mu_r^{**}$-random online partial function has non-empty domain is 0 if $r< 1/4$ and is
\[
\dfrac{4r-1}{4r^2}
\]
if $r\geq1/4$.

\end{proposition}

\begin{proof}
An online partial function $F$ has non-empty domain if and only if there is some $x\in\cs$ such that $f(x\uh n)\neq 2$ for every $n>0$.  Let $A=\{0,1\}$, so that $f(\sigma)\in A$ with probability $p=2r$ for every $\sigma\in\str\setminus\{\epsilon\}$.  Applying Lemma \ref{lem:key}, the $\mu^{**}$-measure of functions $F$ such that there exists $x\in\cs$ with $f(x \uh n) \in A$ for all $n$ equals 0 if $r<1/4$ and equals $\frac{2p-1}{p^2}=\frac{4r-1}{4r^2}$ if $r \geq 1/4$. 

%
%
%
%

\end{proof}

The final step to obtaining a collection of random functions whose ranges are the standard random closed sets is to consider a wider class of measures, namely, computable, symmetric \emph{generalized} Bernoulli measures on $\ccs$.  Such a measure is given by a computable sequence of rationals $\vec{r}=(r_i)_{i\in\omega}$ with $r_i\leq 1/2$ for every $i$ such that for each $n$ and each $\sigma$ of length $n$, $\mu_{\vec{r}}(\sigma 0\mid\sigma)=\mu(\sigma 1\mid\sigma)=r_n\cdot\mu(\sigma)$ and $\mu_{\vec{r}}(\sigma 2\mid\sigma)=(1-2r_n)\mu(\sigma)$.  We can now prove the following.

\begin{theorem}\label{thm-capacity-racs}
Let $\T$ be an computable capacity on $\C(\cs)$ such that there is a computable sequence of rationals $(p_i)_{i\in\omega}$ satisfying
\begin{itemize} 
\item[(i)] for each $n$, $\T(\llb\sigma\rrb)=p_n$ for every $\sigma\in 2^n$, and
\item[(ii)] $\lim_{n\rightarrow\infty}p_n=0$.
\end{itemize}
Then there is a computable, generalized symmetric Bernoulli measure $\mu_{\vec{r}}$ on $\ccs$ such that the ranges of the $\mu_{\vec{r}}^{**}$-random online partial functions are precisely the random closed sets associated with the capacity $\T$. Moreover, in the case that  $\lim_{n\rightarrow\infty}\frac{p_{n+1}}{p_n}=p$ for some $p\in[0,1]$, we have $\lim_{n\rightarrow \infty}r_n=\frac{p}{2}$.
\end{theorem}

\begin{proof}
To obtain the measure $\mu_{\vec{r}}$, we suppose we have a collection of $\mu_{\vec{r}}$-random functions that yield the hitting probabilities $(p_n)_{n\in\omega}$ 
then follow the proof of Theorem \ref{thm:hitting-prob} to recover the values of the sequence $(r_i)_{i\in\omega}$.

Without loss of generality, we can consider the probability of hitting $\llb0^n\rrb$ for each $n$.  By convention, $p_0=\T(\emptyset)=1$.  For $n\geq 0$, to determine the relationship between $p_{n+1}$ and $p_n$, we consider the possible initial values $f(0)$ and $f(1)$ of a representing function $f:\str\setminus\{\epsilon\}\rightarrow\{0,1,2\}$ corresponding to an arbitrary $F\in\F(\cs)$.  Due to our new interpretation of 2s, we only have a total of four cases to consider:

\begin{itemize}
\item[] \emph{Case 1}:  $f(0)\neq 0$ and $f(1)\neq 0$, then $\ran(F)\cap\llb 0^{n+1}\rrb=\emptyset$.
\item[] \emph{Case 2}: If $f(0)=f(1)=0$, which occurs with probability $r_{n+1}^2$, then $\ran(F)\cap\llb 0^{n+1}\rrb\neq\emptyset$ with probability $1-(1-p_n)^2=2p_n-p_n^2$.
\item[] \emph{Case 3}:  $f(i)=0$ and $f(1-i)=1$, which occurs with probability $2r_{n+1}^2$, then $\ran(F)\cap\llb 0^{n+1}\rrb\neq\emptyset$ with probability $p_n$.
\item[] \emph{Case 4}: $f(i)=0$ and $f(1-i)=2$, which occurs with probability $2r_{n+1}(1-2r_{n+1})$, then $\ran(F)\cap\llb 0^{n+1}\rrb\neq\emptyset$ with probability $p_n$.
\end{itemize}
Combining these cases yields
\[
p_{n+1}=(2p_n-p_n^2)r_{n+1}^2+2p_nr_{n+1}^2+2r_{n+1}(1-2r_{n+1})p_n,
\]
which simplifies to
\[
p_{n+1}=2p_nr_{n+1}-p_n^2r_{n+1}^2.
\]
Solving for $r_{n+1}$ yields
\[
r_{n+1}=\frac{p_{n+1}}{p_n(1+\sqrt{1-p_{n+1}})}.
\]
It follows that the capacity induced by the family of $\mu_{\vec{r}}^{**}$-random online partial functions is the capacity $\T$.  Now, the map $\Phi$ that maps a $\mu_{\vec{r}}^{**}$-random online partial function $F$ to its range is still a computable map, as we can effectively determine those basic open neighborhoods $\llb\sigma\rrb$ on which $F$ is undefined.  Then if we let $\nu^{*}$ be the computable measure on $\C(\cs)$ induced by $\Phi$ and $\mu_{\vec{r}}$ (as in the proof of Theorem \ref{thm-fns-caps}), then we will have
\[
\T(Q)=\mu_{\vec{r}}^{**}(\{F\in\F(\cs):\ran(F)\cap Q\neq\emptyset\})=\nu^{*}(\{C\in\C(\cs):\C\cap Q\neq\emptyset\}).
\]
for every $Q\in\C(\cs)$.  Thus, the ranges of the $\mu_{\vec{r}}^{**}$-random online partial functions are the random closed sets associated to $\T$.

%

Lastly, observe that
\[
\lim_{n\rightarrow \infty}r_n=\lim_{n\rightarrow \infty}\frac{p_{n+1}}{p_n(1+\sqrt{1-p_{n+1}})}=\Bigl(\lim_{n\rightarrow\infty}\frac{p_{n+1}}{p_n}\Bigr)\Bigl(\lim_{n\rightarrow\infty}\frac{1}{1+\sqrt{1-p_{n+1}}}\Bigr)=\frac{p}{2}.
\]
\end{proof}

\begin{theorem}
Let $\vec{r}=(r_i)_{i\in\omega}$ be defined by 
\[
r_i=\frac{2/3}{1+\sqrt{1-\Bigl(\frac{2}{3}\Bigr)^i}}.
\]
Then the collection of ranges of the $\mu_{\vec{r}}^{**}$-random online partial functions is equal to the collection of the standard random closed sets.
\end{theorem}

\begin{proof}
Let $\T$ be the capacity associated to the standard random closed sets.  As discussed in Section \ref{sec-functions-to-capacities}, we have $\T(\llb\sigma\rrb)=\Bigl(\frac{2}{3}\Bigr)^n$ for every $n\in\omega$.  Then $\T$ satisfies the conditions of Theorem \ref{thm-capacity-racs}.  By the proof of Theorem \ref{thm-capacity-racs}, if $\mu_{\vec{r}}$ is the computable, symmetric generalized Bernoulli measure on $\ccs$ where 
\[
r_i=\frac{2/3}{1+\sqrt{1-\Bigl(\frac{2}{3}\Bigr)^i}}
\]
for every $i\in\omega$, then the ranges of the $\mu_{\vec{r}}^{**}$-random online partial functions are precisely the standard random closed sets.
\end{proof}

\bibliographystyle{plain}
\bibliography{random}

\begin{thebibliography}{10}

\bibitem{A10}
L.~Axon.
\newblock {\em Algorithmically random closed sets and probability}.
\newblock PhD thesis, University of Notre Dame, 2010.

\bibitem{BCDW07}
G.~Barmpalias, P.~Brodhead, D.~Cenzer, S.~Dashti, and R.~Weber.
\newblock Algorithmic randomness of closed sets.
\newblock {\em J. Logic and Computation}, 17:1041--1062, 2007.

\bibitem{BCRW09}
G.~Barmpalias, P.~Brodhead, D.~Cenzer, J.B. Remmel, and R.~Weber.
\newblock Algorithmic randomness of continuous functions.
\newblock {\em Archive for Mathematical Logic}, 46:533--546, 2008.

\bibitem{BCRW08}
G.~Barmpalias, D.~Cenzer, J.B. Remmel, and R.~Weber.
\newblock $k$-triviality of closed sets and continuous functions.
\newblock {\em Journal of Logic and Computation}, 19:3--16, 2009.

\bibitem{BP12}
L.~Bienvenu and C.P. Porter.
\newblock Strong reductions in effective randomness.
\newblock {\em Theoretical Computer Science}, 459:55--68, 2012.

\bibitem{BCR07}
P.~Brodhead, D.~Cenzer, and J.~Remmel.
\newblock Random continuous functions.
\newblock {\em Electronic Notes in Theoretical Computer Science}, 167:275--287,
  2007.

\bibitem{BCTW11}
D.~Cenzer, P.~Brodhead, F.~Toska, and S.~Wyman.
\newblock Algorithmic randomness and capacity of closed sets.
\newblock {\em Logical Methods in Computer Science}, 6:1--16, 2011.

\bibitem{DK12}
D.~Diamondstone and B.~Kjos-Hanssen.
\newblock \ml randomness and galton-watson processes.
\newblock {\em Annals of Pure and Applied Logic}, 163:519--529, 2012.

\bibitem{DH11}
R.~Downey and D.~Hirschfeldt.
\newblock {\em Algorithmic randomness and complexity}.
\newblock Springer-Verlag, 2011.

\bibitem{ML66}
P.~Martin-Lof.
\newblock The definition of random sequences.
\newblock {\em Information and Control}, 9:602--619, 1966.

\end{thebibliography}

\end{document}